\documentclass[11pt]{amsart}
\usepackage[usenames,dvipsnames]{color}
 \usepackage{amsthm}
\usepackage{amsmath,amssymb}
\usepackage{geometry}
\usepackage{accents}
\usepackage{url} 
\usepackage{hyperref}
\newcommand\cyr
{
\renewcommand\rmdefault{wncyr}
\renewcommand\sfdefault{wncyss}
\renewcommand\encodingdefault{OT2}
\normalfont
\selectfont
}
\DeclareTextFontCommand{\textcyr}{\cyr}

\DeclareMathSymbol{\widehatsym}{\mathord}{largesymbols}{"62}
\newcommand\lowerwidehatsym{%
  \text{\smash{\raisebox{-1.3ex}{%
    $\widehatsym$}}}}
\newcommand\fixwidehat[1]{%
  \mathchoice
    {\accentset{\displaystyle\lowerwidehatsym}{#1}}
    {\accentset{\textstyle\lowerwidehatsym}{#1}}
    {\accentset{\scriptstyle\lowerwidehatsym}{#1}}
    {\accentset{\scriptscriptstyle\lowerwidehatsym}{#1}}
}

\begin{document}

\newtheorem{theorem}{Theorem}
\newtheorem{proposition}[theorem]{Proposition}
\newtheorem{corollary}[theorem]{Corollary}
\newtheorem{lemma}[theorem]{Lemma}
\newtheorem{definition}[theorem]{Definition}
\newtheorem{remark}[theorem]{Remark}
\newtheorem{conjecture}[theorem]{Conjecture}

\newcommand\Z{{\mathbb Z}}
\newcommand\K{{\mathbb K}}
\newcommand\R{{\mathbb R}}
\newcommand\Aff{{\mathbb A}}
\newcommand\Sp{{\mathbb S}}
\newcommand\N{{\mathbb N}}
\newcommand\A{{\mathcal A}}
\newcommand\B{{\mathcal B}}
\newcommand\D{{\mathcal D}}
\newcommand\C{{\mathbb C}}
\newcommand\G{{\mathcal G}}
\renewcommand\O{{\mathcal O}}
\newcommand\cQ{{\mathit{\Delta}}}
\newcommand\cJ{{\mathcal J}}
\newcommand\AV{{\mathcal {AV}}}
\newcommand\hAV{\fixwidehat{\mathcal {AV}}}
\newcommand\hav{\fixwidehat{{AV}}}
\newcommand\wo{\widehat{\otimes}}
\newcommand\F{{\mathcal F}}
\newcommand\V{{\mathcal V}}
\newcommand\cV{{\mathcal V}}
\newcommand\cU{{\mathcal U}}
\newcommand\LL{{\mathcal L}}
\newcommand\J{{\mathcal J}}
\newcommand\hLL{\fixwidehat{\mathcal L}}
\newcommand\hJ{\fixwidehat{J}}
\newcommand\JJ{\fixwidehat{\mathcal{J}}}
\newcommand\dd[1]{\frac{\partial}{\partial #1}}
\newcommand\Der{\textup{Der\,}}
\newcommand\rep{\textup{rep}}
\renewcommand\mod{\textup{mod\,}}
\newcommand\Stab{\textup{Stab}}
\newcommand\Spec{\textup{Spec}}
\newcommand\End{\textup{End}}
\newcommand\Span{\textup{Span}}
\newcommand\ideal{\triangleleft}
\newcommand\gl{\mathfrak{gl}}
\newcommand\m{\mathfrak{m}}
\newcommand\id{\textup{id}}
\newcommand\Q{\Delta}

\title
[A Universal Sheaf of Algebras]
{A universal sheaf of algebras governing representations of vector fields on quasi-projective varieties}
\author{Yuly Billig}
\address{School of Mathematics and Statistics, Carleton University, Ottawa, Canada}
\email{billig@math.carleton.ca}
\author{Colin Ingalls}
\address{School of Mathematics and Statistics, Carleton University, Ottawa, Canada}
\email{cingalls@math.carleton.ca}
\let\thefootnote\relax

\begin{abstract}
We construct a topologically quasi-coherent sheaf of associative algebras which controls a category of finitely generated $AV$-modules over a smooth quasi-projective variety. We establish a local structure theorem, proving that in \'etale charts these associative algebras decompose into a tensor product of the algebra of differential operators and the universal enveloping algebra of the Lie algebra of power series vector fields vanishing at the origin.
\end{abstract}

\maketitle

\section{Introduction}

The motivation for this paper is twofold -- to generalize the notion of $D$-modules and to better understand representation theory of Lie algebras of vector fields.

Conceptually, the Lie algebra $V$ of vector fields on a variety $X$ corresponds to the group of diffeomorphisms of $X$. For this reason vector fields act on many geometric objects associated with $X$. Such actions rarely extend to the algebra $D$ of differential operators. For example, the adjoint action of $V$ on the sections of the tangent bundle does not extend to a $D$-module.

In order to obtain a meaningful representation theory of a Lie algebra, where we can hope to establish structural results, we need to impose constraints on the class of modules we consider. Such a constraint may be a certain finiteness condition, i.e.~considering modules that are finite-dimensional or having finite-dimensional weight spaces.

By a theorem of David Jordan \cite{J1}, \cite{J2} (see also \cite{BF}), Lie algebras of vector fields on smooth irreducible affine varieties are simple. In general, such Lie algebras have
no non-zero semisimple or nilpotent elements \cite{BF}, thus the classical machinery of roots and weights is not applicable for these simple Lie algebras. Instead, we consider the category of $AV$-modules, where we impose a supplementary condition that in addition to the action of the Lie algebra $V$ of vector fields, modules $M$ in this category admit an action of the commutative algebra $A$ of functions on $X$, with the two actions compatible via the Leibniz rule:
$$\eta \cdot f \cdot m = f \cdot \eta \cdot m + \eta(f) \cdot m, \ \ \text{for \ } 
\eta\in V, f\in A, m\in M.$$
$D$-modules satisfy an additional axiom $f \cdot \eta \cdot m = (f \eta) \cdot m$, which we do not impose in the category of $AV$-modules.

The notion of $AV$-modules was instrumental in the proof of the classification theorem for simple $V$-modules with finite-dimensional weight spaces over a torus \cite{BF0} or an affine space \cite{GS}, \cite{XL}, where it was shown that cuspidal $V$-modules possess covers in the category of cuspidal $AV$-modules.

From the geometric perspective, spaces of sections of tensor fields on $X$ are $AV$-modules, very few of which happen to be $D$-modules.

When $X$ is an affine variety, $AV$-modules may be described as representations of an associative algebra $AV$, which is the smash product of a Hopf algebra $\cU(V)$, the universal enveloping algebra of $V$, with its module $A$:
$$AV = A \# \cU(V).$$
In our previous work \cite{BIN} we investigated the structure of $AV$-modules on affine space $\Aff^N$. The key to understanding these modules was the following structure theorem for the algebra $AV$ (see also \cite{XL}):
$$AV(\Aff^N) \cong D(\Aff^N) \otimes_K \cU(\LL_+).$$
Here $\LL_+$ is the Lie algebra of vector fields on $\Aff^N$ vanishing at the origin and $D(\Aff^N)$ is the algebra of differential operators on the affine space.  

This result suggests that locally $AV$-modules are given by two ingredients -- a $D$-module, and a compatible action of $\LL_+$. As a result, there are visible parallels between theories of $AV$-modules and $D$-modules. If we look at the $AV$-modules that have been studied so far in the literature \cite{BFN}, we see that the gauge modules generalize $D$-modules given by vector bundles with flat connections, while Rudakov modules generalize $D$-modules of delta functions. The $\LL_+$ ingredient gives an additional flavour to $AV$-modules, making this category richer than the category of $D$-modules.




If we replace $\Aff^N$ with an arbitrary smooth irreducible affine variety, the above decomposition will not hold in general. In the present paper we show that local analogues of this tensor product decomposition still hold if we take appropriate completions. Let us describe our main results.
The associative algebra $AV = A \# \cU(V)$ has a Lie subalgebra $A \# V$. First we construct its completion: the Lie algebra $\hJ$ of jets of vector fields. For this we consider the kernel $\Delta$ of the multiplication map $A \otimes_K A \rightarrow A$ and set 
$$\hJ = \varprojlim\limits_n A \# V / (\Delta^n \otimes_A V).$$
Then we define the completion $\hav$ as
$$\hav = A \# \cU (\hJ) / \left< f \cdot s - fs \, |\, f \in A, s \in \hJ \right>,$$
where $f \cdot s$ is the product in  $A \# \cU (\hJ)$ and $fs$ is the product in $\hJ$.

A variety $X$ has an atlas of \'etale charts $\{ U_i \}$ (see Section 2 for details). Each \'etale chart comes with uniformizing parameters $\{ x_1, \ldots, x_N \}$ which trivialize the Lie algebra of vector fields: $$V(U) = \mathop\bigoplus\limits_{i=1}^N A(U) \dd{x_i}.$$

Locally in an \'etale chart we establish structure theorems for the Lie algebra of jets $\hJ$ and for the associative algebra $\hav$.

\begin{theorem}
Let $U \subset X$ be an \'etale chart. Then
$$\hJ (U) \cong V(U) \ltimes (A(U) \wo \LL_+),$$
$$\hav(U) \cong D(U) \mathop\otimes\limits_A \cU_A(A(U) \wo \LL_+).$$
Here $\LL_+$ is the Lie algebra of vector fields on $\Aff^N$ vanishing at the origin and $A(U) \wo \LL_+$ is a completion of the tensor product.
\end{theorem} 

When $X$ is a quasi-projective variety, we need to replace global objects like the Lie algebra of vector fields with the corresponding sheaves. If we sheafify the algebras $AV$, the resulting sheaf will fail, in general, to be quasi-coherent. An additional benefit of passing to the completion is that the sheaf $\hAV$ of completed algebras becomes topologically quasi-coherent (Theorem \ref{quasicoherent}). We also present an explicit localization formula (Proposition \ref{localizationformula}).

Motivated by the completion construction for algebras $AV$, we proposed to Emile Bouaziz and Henrique Rocha a conjecture about finitely generated $AV$-modules, which they successfully proved:
\begin{theorem} \label{thm:BR22} (\cite{BR})
Let $X$ be a smooth irreducible affine algebraic variety. Let $M$ be an $AV$-module on $X$ which is finitely generated over $A$. Then there exists $k\in \N$ depending on $rk_A(M)$, such that 
$\Delta^k \otimes_A V$ annihilates $M$.
\end{theorem}
This means that on each $AV$-module $M$ which is finitely generated as an $A$-module, the action of the completion trivializes, with the tails of a high enough order acting trivially on $M$.

If we want to study the sheafs of $AV$-modules, completion of algebra $AV$ becomes inevitable.
Our Theorem \ref{bundle} shows that when we perform a change of coordinates, an element of 
Lie algebra $\LL_+$ transforms into a power series from $A \wo \LL_+$. At the same time, we are still capturing the essential part of the representation theory, as the most important $AV$-modules, 
guage modules and Rudakov modules admit the action of $\hAV$ since each of them is annihilated by $\Delta^k \otimes_A V$ for some $k$. Thus $\hAV$ is the universal sheaf of algebras governing the theory of $AV$-modules over a quasiprojective variety.


\

The structure of the paper is as follows. In Section 2 we discuss \'etale charts on $X$. In Section 3 we review the notion of a Rinehart pair and the construction of its weak and strong enveloping algebras.
In Section 4 we develop the theory of jets of functions, which we then apply in Section 5 to construct the jets of vector fields and the completion of the associative algebra $AV$. We also establish our main results in Section 5, proving structure theorems for $\hJ$ and for $\hav$ in \'etale charts.
\subsection*{Acknowledgements}
Research of both authors is supported with grants from the
Natural Sciences and Engineering Research Council of Canada.

\section{\'Etale charts}

Let $K$ be an algebraically closed field of characteristic 0, and let $X$ be a smooth irreducible quasi-projective algebraic variety of dimension $N$. We denote by $\A$ the sheaf of regular functions on $X$,
by $\V$ the sheaf of vector fields, and by $\D$ the sheaf of differential operators on $X$.

Let us describe these sheaves in more detail. Let $U$ be an affine Zariski open subset of $X$, and set
$A = \A(U)$, $V = \V(U)$, $D = \D(U)$ to be the spaces of sections of the respective sheaves over $U$. Here $A$ is a commutative algebra, $V$ is a Lie algebra, and $D$ is an associative algebra.
These algebras are related in the following ways: $V = \Der (A)$ and $D$ is the subalgebra in $\End_{K} (A)$ generated by $A$ (acting on itself by multiplication) and $V$ (acting on $A$ by derivations).

Furthermore, $X$ has an atlas $\{ U_i \}$ of \emph{\'etale charts}, $X = \mathop\bigcup\limits_i U_i$ (see e.g. \cite{BF}).

\begin{definition}
An affine open subset $U \subset X$ is called an \'etale chart if there exist functions
$x_1, \ldots, x_N \in A = \A(U)$ such that

(1) the set $\{x_1, \ldots, x_N\}$ is algebraically independent, that is $K[x_1, \ldots, x_N] \subset A$,

(2) every $f \in A$ is algebraic over $K[x_1, \ldots, x_N]$,

(3) the derivations $\dd{x_1}, \ldots, \dd{x_N}$ of  $K[x_1, \ldots, x_N]$ extend to derivations of $A$. 
\end{definition}

We will call such $\{x_1, \ldots, x_N\}$ \emph{uniformizing parameters} on $U$. Since $A$ is algebraic over 
$K[x_1, \ldots, x_N]$, an extension of $\dd{x_i}$ to $A$ is unique. The vector fields  $\dd{x_1}, \ldots, \dd{x_N}$ commute.

\begin{lemma} (\cite[Theorem III.6.1]{Mu}, \cite{BF})
\label{Mumford}
Let $U$ be an \'etale chart of $X$ with uniformizing parameters ${x_1, \ldots, x_N}$.
Let  $A = \A(U)$, $V = \V(U)$, $D = \D(U)$. Then

(1)
$$ V = \mathop\bigoplus\limits_{i=1}^N A \dd{x_i},$$

(2)
$$ D = \mathop\bigoplus\limits_{k \in \Z_{\geq 0}^N} A \partial^k,$$
where for $k = (k_1, \ldots, k_N)$ we set $\partial^k = 
\left( \dd{x_1} \right)^{k_1} \ldots \left( \dd{x_N} \right)^{k_N}$.

(3)
$$\Omega^1 (U) =  \mathop\bigoplus\limits_{i=1}^N A d x_i$$
with the differential given by
$$df = \sum\limits_{i=1}^N \frac {\partial f}{\partial x_i} dx_i  \text{ \ for \ } f \in A.$$

(4) The map $(x_1, \ldots, x_N): \ U \rightarrow \Aff^N$ is \'etale.

(5) For every $P \in U$ functions $t_1 = x_1 - x_1(P), \ldots, t_N = x_N - x_N(P)$ generate 
$\m_P / \m^2_P$.
\end{lemma}

Part (1) of this lemma is well-known (see e.g., \cite{BF}), while part (2) follows from (1) and the fact that 
$D$ is generated by $A$ and $V$.
Consider the morphism $\phi: U \to \Aff^N$. Condition (1) of the definition shows that $\phi$ is dominant, condition (2) shows that $\phi$ is quasifinite and lastly, condition (3) says that $\phi$ is smooth. So we can conclude (4), that $\phi$ is \'etale. The other properties are properties of \'etale maps.

\section{Rinehart enveloping algebras}

Fix an affine open set $U$ and consider the commutative algebra $A = \A(U)$, and the Lie algebra of vector fields $V = \V(U)$. 
We point out that $V$ is an $A$-module and $A$ is a $V$-module. This is an example of a {\it Rinehart pair} \cite{Ri}.

\begin{definition}
A Rinehart pair $(A, G)$ consists of a commutative associative unital algebra $A$, and a Lie algebra $G$, with $G$ being an $A$-module,
and $G$ acting on $A$ by derivations, satisfying
$$(a g) (b) = a (g(b)),  \ \ \ \ {\text for \ } a, b \in A, \, g \in G,$$
with the two actions being compatible via the Leibniz rule:
$$[g_1, a \cdot g_2] = g_1 (a) \cdot g_2 + a \cdot [g_1, g_2], \ \ \ \ {\text for \ } a \in A, \, g_1, g_2 \in G.$$ 
\end{definition}
The action of $G$ on $A$ yields a homomorphism of Lie algebras $\rho: \, G \rightarrow V = \Der A$, which is called the {\it anchor map}.

For a Rinehart pair $(A, G)$ we can define two associative enveloping algebras, the {\it weak Rinehart enveloping algebra} $U(A,G)$ 
and the {\it strong Rinehart enveloping algebra} $D(A,G)$. 

The weak Rinehart enveloping algebra 
$$U(A,G) = A\# \cU(G)$$
is the smash product of the universal enveloping algebra $\cU(G)$, viewed as a Hopf algebra,
with its module $(A, \rho)$. As a vector space it is $A \otimes_K \cU(G),$
with $A$ and $\cU(G)$ being subalgebras in $U(A,G)$, and the commutation relation between $A$ and 
$\cU(G)$ given by the Leibniz rule:
$$ g \cdot a = a \cdot g + \rho(g)(a), \text{\ for \ } g \in G, \ a \in A.$$
Here $\cdot$ denotes the product in $U(A,G)$. We will also use the notation $a \cdot g = a \# g$. 

The strong Rinehart enveloping algebra $D(A,G)$ is the quotient of the weak Rinehart enveloping algebra $U(A,G)$ by the
ideal generated by the elements $\{ a \cdot g - ag \, | \, a\in A, g\in G\}$. Here $a \cdot g$ denotes the product in $U(A,G)$,
while $ag \in G$ is obtained via the $A$-module structure on $G$.

 When $G = V$, we will denote the weak Rinehart algebra $U(A, V)$ by $AV$. In this case the strong Rinehart algebra $D(A,V)$ is
the algebra $D$ of differential operators \cite{Ba}.

We point out that the subspace $A\# G = A \otimes_K G$ in $U(A, G)$ is a Lie subalgebra with the Lie bracket
\begin{equation}
\label{smash}
[a_1 \# g_1, a_2 \# g_2] = a_1 \rho(g_1)(a_2) \# g_2 - a_2 \rho(g_2) (a_1) \# g_1 + a_1 a_2 \# [g_1, g_2], 
\end{equation}
for $a_1, a_2 \in A, g_1, g_2 \in G$.
There is a natural left $A$-action on $A\# G$ (on the left tensor factor), and there is a natural action of $A \# G$ on $A$ by derivations.

It is easy to check that the following lemma holds:
\begin{lemma}
For a Rinehart pair $(A, G)$, the pair $(A, A\#G)$ is again a Rinehart pair. 
\end{lemma}

The next lemma exhibits a relationship between strong and weak Rinehart enveloping algebras:

\begin{lemma}
$U(A, G) \cong D(A, A\#G)$.
\end{lemma}
\begin{proof}
It can be seen that both algebras are quotients of the tensor algebra $T(A \oplus G)$ subject to the same relations:
\begin{align*}
a_1 \cdot a_2 &= a_1 a_2, \\
g_1 \cdot g_2 - g_2 \cdot g_1 &= [g_1, g_2], \\
g_1 \cdot a_1 &= \rho(g_1) (a_1) + a_1 \cdot g_1,
\end{align*}
for $a_1, a_2 \in A$, $g_1, g_2 \in G$.
\end{proof}

%

\begin{definition}
A Lie algebroid on $X$ is a topologically quasi-coherent (see Definition \ref{topq} below) 
sheaf $\G$ such that for any open set $U$ the pair $(\A(U), \G(U))$ is a Rinehart pair, and the restriction maps are 
the morphisms of the corresponding Rinehart pairs.
\end{definition}

Given a Lie algebroid $\G$, we can sheafify the constructions of weak and strong Rinehart enveloping algebras and obtain sheaves of associative 
algebras $U(\A, \G)$ and $\D (\A, \G)$. When $\G$ is the tangent bundle $\V$, we will denote the sheaf $U(\A, \V)$ by $\AV$, while the sheaf
$\D(\A, \V)$ is the sheaf $\D$ of differential operators on $X$. 

We point out that the sheaf $\AV$ is naturally a left $\A$-module, since for each affine open set $U$ the algebra $AV$ is a left $A$-module, and this structure is compatible with the restriction maps.
Yet, $\AV$ is not, in general, a quasi-coherent sheaf.  Note that,
for an affine open $U$ we have
$\AV(U) = \A(U) \otimes_K \V(U),$ which is naturally a module over $\A(U)\otimes_k\A(U)$ and so naturally gives a quasi-coherent on the product $X\times X$.

Let us recall a definition of a quasi-coherent sheaf:

\begin{theorem} (see \cite[Theorem-Definition III.1.3]{Mu}, \cite[Prop.~5.2.2]{Kempf})
\label{localization}
An $\A$-module sheaf $\F$ is quasi-coherent if and only if there exists an affine open cover $
X = \mathop\bigcup\limits_i U_i$ such that for every non-zero $g \in \A(U_i)$ we have that the multiplication map below is an isomorphism:
$$\A (U_i)_g \otimes_{\A(U_i)} \F (U_i) \rightarrow \F (U_i^g),$$
where $U_i^g$ is $U_i$ without the zero locus of $g$.  
\end{theorem}

We see that for the sheaf $\AV$ the space in the left hand side is
$$\A (U)_g \otimes_{\A(U)} \AV (U) \cong \A (U)_g \otimes_K \cU (\V(U)),$$
while the space in the right hand side is
$$\AV (U^g) \cong \A (U)_g \otimes_K \cU (\V(U^g)),$$
and in general the natural map between these two spaces is not surjective. Thus, the sheaf $\AV$ is not quasi-coherent. 
It is well-known that the sheaf $\D$ of differential operators is quasi-coherent.
One of the main goals of this paper is to define a completion of $\AV$ which will become topologically quasi-coherent. 

\section{Jets}

Let $U$ be an \'etale chart with the uniformizing parameters $\{ x_1, \ldots, x_N \}$ and let $f \in A = \A(U)$. 
Consider the map $A \rightarrow A[[t_1,\ldots,t_N]]$:
{${}^1$}\footnote{{${}^1$}In the published version of this paper, we made a mistake of identifying $A[[t]]$ with
$A \otimes K[[t]]$. Tensor products certainly do not commute with the inverse limits: 
$A[[t]] = \lim\limits_{\longleftarrow}A \otimes K[t]/\left<t^m\right>$ and $A \otimes K[[t]] = A \otimes \lim\limits_{\longleftarrow} K[t]/\left<t^m\right>$. There is a proper inclusion $A \otimes K[[t]] \subset A[[t]]$.}:
\begin{equation}
\label{Taylor_xt}
f \mapsto \sum_{m \in \Z^N_{\geq 0}} \frac{1}{m!} \frac{\partial^m f}{\partial x^m} t^m.
\end{equation}
To simplify notations we will denote $A[[t_1,\ldots,t_N]]$ by $A[[t]]$.

The image of $f$ under this map is called the {\it jet} of a function $f$. The jet of $f$ can be seen to be a formal Taylor expansion of $f(x+t)$. It is well-known from the first-year calculus (and can be easily checked directly) that this map is a homomorphism of commutative algebras. To simplify notations, we will actually denote the right hand side of (\ref{Taylor_xt}) by $f(x+t)$. In the sum $x+t$ the first variable will denote the point of expansion in a series, and all subsequent variables, as in $x+t+s$ will be formal parameters. We emphasize that expressions $f(x+t)$ and $f(t+x)$ are not the same.

Viewing $A[[t]]$ as an $A$-module, we extend the above map to another homomorphism of commutative algebras
$$j: \, A\otimes_K A \rightarrow A[[t]],$$
where 
\begin{equation}
\label{jet}
j (g \otimes f) = \sum_{m \in \Z^N_{\geq 0}} \frac{1}{m!} g(x) \frac{\partial^m f}{\partial x^m} t^m
\, = g(x) f(x+t).
\end{equation}

\begin{lemma}
The map $j: \, A\otimes_K A \rightarrow A[[t]]$ is injective.
\end{lemma}
\begin{proof}
Let $P$ be a point in $U$. Then representation of functions by their Taylor series at $P$ is an injective homomorphism $A \rightarrow K[[t]]$:
$$f  \mapsto \sum_{m \in \Z^N_{\geq 0}} \frac{1}{m!} \frac{\partial^m f}{\partial x^m}(P) t^m.$$ 
We shall denote the right hand side above by $f(P+t)$.

Since $A \otimes A$ is the algebra of functions on $U \times U$, for a pair of points $P_1, P_2 \in U$
we have an injective homomorphism  $A \otimes A \rightarrow K[[t, s]]$:
$$g \otimes f \mapsto g(P_1 + t) f(P_2 + s).$$
Since $P_1, P_2$ are arbitrary, we can choose them to be the same point $P$ and we get an
injective map $g \otimes f \mapsto g(P + t) f(P + s)$. Taking $P$ to be a generic point, we get 
an injective homomorphism $A \otimes A \rightarrow A[[t, s]]$:
$$g \otimes f \mapsto g(x + t) f(x + s).$$
Making a substitution $s = t + r$, where $r$ is a new formal parameter, we get an injective map
$A \otimes A \rightarrow A[[t, r]]$:
$$g \otimes f \mapsto g(x + t) f(x + t + r) = 
\sum_{m \in \Z^N_{\geq 0}} \frac{1}{m!} g(x+t) \frac{\partial^m f}{\partial x^m}(x+t) r^m.$$
However, the map (\ref{Taylor_xt}) is an isomorphism to its image, and thus setting $t = 0$ in the above expression, we still have an injective homomorphism $j: \, A\otimes_K A \rightarrow A[[r]]$.
\end{proof}

The {\it algebra of jets} will be the completion $A \wo A$ of  $A \otimes A$ in the topology of $A[[t]]$. 
Let us discuss this definition in detail.  

Consider the multiplication map $A \otimes_K A \rightarrow A$. Let $\Q \vartriangleleft A \otimes_K A$ be the kernel of this map. 
 
We define the algebra of jets of functions as the following completion of a commutative algebra $A \otimes_K A$:
 $$A \wo A = \varprojlim\limits_k (A \otimes_K A) / \Q^k.$$
 Elements of $A \wo A$ can be written as infinite series
 $\sum\limits_{k=0}^\infty u_k$ with $u_k \in \Q^k$ (assuming $\Q^0 = A \otimes A$). Two series 
 are equal,
 $\mathop\sum\limits_{k=0}^\infty u_k = \mathop\sum\limits_{k=0}^\infty w_k,$
 if and only if $\sum\limits_{k=0}^m (u_k - w_k) \in \Q^{m+1}$ for all $m = 0, 1, 2, \ldots$. 

\begin{lemma}
Algebra $A\otimes_K A$ is separated with respect to ideal $\Q$, that is 
$\mathop\bigcap\limits_{k=1}^\infty \Q^k = 0$.
\end{lemma}
\begin{proof}
  This follows from the fact that $j(\Q^k) \subset t^k A[[t]]$,
where by $t^k A[[t]]$ we denote the $k$-th power of the ideal generated by
 $\left< t_1,\ldots, t_N \right>$.
  Alternatively, one can use \cite[Cor.~10.18]{AM}.
\end{proof}

For $f \in A$ we set $\delta(f) = f \otimes 1 - 1 \otimes f \in \Q$.
 \begin{lemma}
 \label{Qgen}
 (1) For $f, g \in A$ we have 
 $$\delta(fg) = (f \otimes 1) \delta(g) + \delta(f) (1 \otimes g) \text{\ \ [Leibniz rule].}$$
 
 
 (2) The set $\{\delta(f) \, | \, f \in A \}$ generates $\Q$ as a left (or right) $A$-module.
 
 (3) The set $\{ \delta(f_1) \ldots \delta(f_k)  \, | \, f_1, \ldots, f_k \in A \}$ generates $\Q^k$ as a left (or right) $A$-module.
 
 (4) $\Q / \Q^2 \cong \Omega^1 (U)$.
 \end{lemma}
 \begin{proof}
 Verification of part (1) is straightforward. To see that (2) holds, consider $\sum_i g_i \otimes f_i \in \Q$. Then
$$- \sum\limits_i g_i \otimes f_i =  \sum\limits_i g_i f_i \otimes  1 - \sum\limits_i g_i \otimes f_i
=  \sum\limits_i (g_i \otimes 1) \delta(f_i) .$$
Part (3) may be easily established by induction using (2) and (1). 

The result (4) is well known, see for instance~\cite[\S 25]{Matsumura}. 
\end{proof}

The proof of the following Lemma is a straightforward calculation.
\begin{lemma} Let $s \in \Z_{\geq 0}^N$. Then
$$j \left( \delta(x)^s \right) = 1 \otimes (-t)^s.$$
\end{lemma}
 
\begin{proposition}
\label{quasiTaylor}
Let $U$ be an \'etale chart of $X$ with uniformizing parameters ${x_1, \ldots, x_N}$.
Let $A = \A(U)$. Then for any $f \in A$ the following identities hold in $A \wo A$:
$$1 \otimes f = \sum_{m \in \Z_{\geq 0}^N} \frac{(-1)^m}{m!} \left( \frac{\partial^m f}{\partial x^m}   
\otimes 1 \right) \delta(x)^m .$$
$$f \otimes 1 = \sum_{m \in \Z_{\geq 0}^N} \frac{1}{m!} \left( 1 \otimes \frac{\partial^m f}{\partial x^m} \right) \delta(x)^m .$$
\end{proposition} 
\begin{remark}
These formulas are just the classical Taylor formula. If we write $f \otimes g$ as $f(x) g(a)$ then the second formula will read
$$ f(x) =  \sum_{m \in \Z_{\geq 0}^N} \frac{1}{m!} \, \frac{\partial^m f} {\partial x^m}(a) \, (x-a)^m.$$
\end{remark}

\begin{proof}
The proofs of the two formulas are identical, so let us prove the second formula. 
Since $f \otimes 1 - 1 \otimes f \in \Q$, we can apply Lemma \ref{Qgen} (4) and Lemma \ref{Mumford} (3) to write
$$f \otimes 1 = 1 \otimes f + \sum_{i=1}^N (1 \otimes h_i) \delta(x_i) \ \mod \, \Q^2 .$$
Iterating the same procedure for $\Q^2$ and the higher powers of $\Q$, i.e. using Hensel's Lemma, we get a series expansion
in $A \wo A$:
\begin{equation}
\label{unknown}
f \otimes 1 = 1 \otimes f + \sum_{m \in \Z^N_{\geq 0} \backslash \{ 0 \} }  (1 \otimes h_m) \delta(x)^m,
\text{ \ for some \ } h_m \in A.
\end{equation}
We need to show that $$h_m = \frac{1}{m!}  \, \frac{\partial^m f}{\partial x^m}.$$

Consider the action of vector fields on $A \otimes_K A$ by their action on the first tensor factor:
$$\eta (f \otimes g) = \eta(f) \otimes g, \text{ \ for \ } \eta \in V.$$
Since $V$ acts by derivations of $A \otimes_K A$, we get that $\eta(\Q^{k+1}) \subset \Q^k$, which means that the action of $V$ is continuous and thus extends to $A \wo A$.

Apply to both sides of (\ref{unknown}) a composition of vector fields $\left( \frac{\partial}{\partial x} 
\right)^m$, followed by the multiplication map $A \otimes_K A \rightarrow A$. Then the left hand side will yield $\frac{\partial^m f}{\partial x^m}$ and the right hand side will yield $m!\, h_m$.
This completes the proof of the Proposition.
\end{proof}

As a corollary we get the following description of the ideals $\Q^k$.

\begin{lemma}
\label{idealpower}
The ideal $\Q^k$ is the pre-image of $t^k A[[t]]$ under the homomorphism \break
$j: \ A \otimes_K A \rightarrow A[[t]]$.
\end{lemma}
\begin{proof}
Since $j(\Q) \subset t A[[t]]$, we get that $j(\Q^k) \subset  t^k A[[t]]$.
Suppose now $j \left( \sum_i g_i \otimes f_i \right) \in t^k A[[t]]$. Then for all $m \in Z^N_+$ with $|m| < k$ we have $\sum_i g_i  \frac{ \partial^m f_i}{\partial x^m} = 0$.
Then by Proposition \ref{quasiTaylor} we have $\sum_i g_i \otimes f_i \in \widehat{\Q}^k$. By \cite[Corollary 10.4]{AM}
the map $A \otimes A \rightarrow A \wo A / \widehat{\Q}^k$ has kernel $\Q^k$. Since 
$\sum_i g_i \otimes f_i$ is in the kernel of this map, it indeed belongs to $\Q^k$. This completes the proof of the Lemma.
\end{proof}
The structure of the {\it algebra of $k$-jets} is given by the following Corollary:
\begin{corollary}
We have the following isomorphisms of algebras:
$$A \otimes A /\Q^{k+1} \cong A \wo A/\widehat{\Q}^{k+1}
\cong A[[t]]/t^{k+1}A[[t]].$$
\end{corollary}
\begin{proof}
The isomorphism $A \otimes A /\Q^{k+1} \cong A \wo A/\widehat{\Q}^{k+1}$ follows from 
\cite{AM} Corollary 10.4, and the isomorphism $A \wo A/\widehat{\Q}^{k+1}
\cong A[[t]]/t^{k+1}A[[t]]$ follows from Lemma \ref{idealpower}.
\end{proof}
\begin{corollary}
The map $j$ induces an isomorphism of algebras $A \wo A$ and $A[[t]]$.
\end{corollary}

\section{Completion of $\AV$}

Next let us define the jets of vector fields. The Lie algebra of {\it polynomial jets of vector fields} is
$$J = (A \otimes_K A) \otimes_A V = A \# V.$$
Lie bracket in $J$ is given by (\ref{smash}).
\begin{equation}
[a_1 \# g_1, a_2 \# g_2] = a_1 \rho(g_1)(a_2) \# g_2 - a_2 \rho(g_2) (a_1) \# g_1 + a_1 a_2 \# [g_1, g_2], 
\end{equation}
for $a_1, a_2 \in A, g_1, g_2 \in G$.

The commutative algebra $A \otimes_K A$ is a $V$-module as a tensor product of two $V$-modules. It is easy to see that the multiplication
map $A \otimes_K A \rightarrow A$ is a homomorphism of $V$-modules and thus its kernel $\Q$ is a $V$-submodule in $A \otimes_K A$.
It follows that its powers $\Q^k$ are also $V$-submodules. It is then easy to see that the spaces $J_k = \Q^{k+1} \otimes_A V$ are Lie ideals in
$J = A \# V$.
\begin{definition}\label{defn:J}
The Lie algebra of jets of vector fields is 
$$\hJ = \varprojlim\limits_k \ (A \# V) / J_k.$$
\end{definition}

By construction, Lie algebra $\hJ$ has a quotient of $m$-jets: $J^m = (A \# V) / J_m.$


Let $X$ be a smooth variety, write $\cQ$ for the ideal of the diagonal of $X\times X$.  Let $X \wo X $ be the formal completion of $X\times X$ along the diagonal.
Write $p_1,p_2: X\times X \to X$ be the projections.
In~\cite{Vakil}, the sheaf of $m$-jets of vector fields is defined by 
$$\cJ^m\cV : = p_{1*} (\O_{X\times X}/\cQ^{m+1} \otimes_{X\times X} \pi_2^* \cV).$$
Similarly we define
$$\JJ :=  p_{1*} (\O_{X \wo X} \otimes_{X\times X} \pi_2^* \cV).$$
This construction agrees with Definition \ref{defn:J} locally and so is the sheaf $\JJ$ of jets of vector fields on $X$.
The sheaf $\JJ$ is a Lie algebroid on $X$. The anchor map is induced from the multiplication map
$A \# V \rightarrow V$, where $f \# \eta \mapsto f\eta$. Since $\Q \otimes_A V$ is in the kernel of the multiplication map, the anchor map naturally extends to $\hJ \rightarrow V$. 

\begin{definition}
(a) We define an associative algebra $\hav$ as
$$\hav = D(A, \hJ) .$$ 

(b) Algebra $\hav$ admits quotients 
$$AV^m = D(A, J^m) .$$ 
\end{definition}


This construction yields sheaves $\hAV$ and $\AV^m$ of associative algebras on $X$ as shown in~\cite{BB}, see also~\cite{MM}.

 We define the category of weak finite $\V$-modules
$\mathrm{weak}_{\mathrm{fin}} \V$ to have objects that are
finitely generated $\A$-modules that are compatibly $\V$-modules,
so locally $\hav$-modules by Theorem~\ref{thm:BR22}.
We can also define the category of strong $\JJ$-modules $\mathrm{strong}_{\mathrm{fin}} \JJ$ to be finitely generated over $\A$
representations of the Lie algebroid, in other words, $D(\A,\JJ)$-modules.
  \begin{corollary}
    There is an equivalence of categories
    $$\mathrm{weak}_{\mathrm{fin}} \V \simeq \mathrm{strong}_{\mathrm{fin}} \JJ.$$
    \end{corollary}

 We would like to describe the restrictions of $\JJ$ and $\hAV$ to an \'etale chart. For this we will need to introduce a Lie algebra $\LL_+$. 
The Lie algebra $\LL = \Der K[X_1, \ldots, X_N]$ has a natural $\Z$-grading by total degree:
$$\LL = \LL_{-1} \oplus \LL_0 \oplus \LL_1 \oplus \LL_2 \oplus \ldots.$$
We point out that $\LL_0 \cong \gl_N (K)$. 
We define $\LL_+$ as its subalgebra corresponding to the non-negative part of this grading:
$$\LL_+ = \mathop\bigoplus\limits_{i = 0}^\infty \LL_i ,$$
and set 
$$\LL_{\geq k} = \mathop\bigoplus\limits_{i = k}^\infty \LL_i .$$
We will also need a completed tensor product:
$$A \wo \LL_+ = \varprojlim\limits_{k} A \otimes \left( \LL_+ / \LL_{\geq k} \right).$$
 
 In the theorem below, we consider a semidirect product of $A \wo \LL_+$ with $V$, where $V$ naturally acts on $A$ and commutes with $1 \wo \LL_+$.
 
 The Lie algebras $J$ and $\hJ$ have a subalgebra $1 \# V$, which is naturally isomorphic to $V$. Over an \'etale chart, there is a second subalgebra isomorphic to $V$: 
$$ \sum_{i=1}^N A \# \dd{x_i}.$$
It is this subalgebra that appears in the next theorem.

\begin{theorem}
\label{isojet}
Let $U$ be an \'etale chart of $X$ with uniformizing parameters ${x_1, \ldots, x_N}$.

(a) Let $A = \A(U)$, $V = \V(U)$, $\hJ = \JJ(U)$. Then
$$\hJ \, \cong \, V \ltimes (A \wo \LL_+),$$
with the isomorphism maps $\varphi: \, \hJ \rightarrow V \ltimes (A \wo \LL_+)$
and $\psi: \,  V \ltimes (A \wo \LL_+) \rightarrow \hJ$ 
being homomorphisms of left $A$-modules given by
\begin{equation}
\label{phiJ}
\varphi\left(g \otimes f \dd{x_i} \right) =g f \dd{x_i} 
+ \sum_{k \in \Z_{\geq 0}^N \backslash \{ 0 \}} \frac{1}{k!} \, g \frac{\partial^k f}{\partial x^k} \otimes X^k \dd{X_i},
\end{equation}
\begin{equation}
\label{psiV}
\psi\left(g \dd{x_i}\right) = g \otimes \dd{x_i} , 
\end{equation}
\begin{equation}
\label{psiLp}
\psi\left(g \otimes X^m \dd{X_i} \right) = (-1)^{m} (g \otimes 1) \delta(x)^m \dd{x_i},
\end{equation}
and extended to completions by continuity.

(b) Let $J^m = \J^m(U)$ be $m$-jets of vector fields on $U$. Then
$$J^m \, \cong \, V \ltimes (A \otimes (\LL_+/\LL_{> m})).$$
\end{theorem}
\begin{proof}
The definition of $\varphi$ is very similar to the definition (\ref{jet}) of the jet map $j$. Since
$j(\Delta^k) \subset t^k A[[t]]$, we conclude that the map $\varphi$ is compatible with completions. The fact that $\psi$ is compatible with completions is immediate from its definition.

Let us verify that $\varphi$ and $\psi$ are homomorphisms of Lie algebras and are inverses of each other.

To show that $\varphi$ is a homomorphism, we need to check that 
$$\left[ \varphi \left(g_1 \otimes f_1 \dd{x_i} \right), \varphi \left(g_2 \otimes f_2 \dd{x_j} \right)
\right] 
= \varphi \left( \left[ g_1 \otimes f_1 \dd{x_i}, g_2 \otimes f_2 \dd{x_j} \right] \right).$$
To simplify the calculations, we will carry them out in two stages. First we will assume $g_1 = g_2 = 1$.
We rewrite the definition of $\varphi$ using the Taylor's formula:
$$\varphi\left(1 \otimes f \dd{x_i} \right) = f \dd{x_i} + 
\left( f(x + X) - f(x) \right) \dd{X_i}.$$
Then
\begin{align*}
&\left[ \varphi \left(1 \otimes f \dd{x_i} \right), \varphi \left(1 \otimes g \dd{x_j} \right) \right] 
= \left[ f \dd{x_i}, g \dd{x_j} \right] \\
&{\hskip 2cm}+ \left[  f \dd{x_i}, (g(x+X) - g(x)) \dd{X_j} \right] 
- \left[  g \dd{x_j}, (f(x+X) - f(x)) \dd{X_i} \right] \\
&{\hskip 2cm}+ \left[  (f(x+X) - f(x)) \dd{X_i} , (g(x+X) - g(x)) \dd{X_j} \right] \\
&= \left[ f \dd{x_i}, g \dd{x_j} \right] \\
&{\hskip 0.5cm}+ f(x) \left( \frac{\partial g}{\partial x_i} (x+X) - \frac{\partial g}{\partial x_i} (x) \right) \dd{X_j} 
-  g(x) \left( \frac{\partial f}{\partial x_j} (x+X)
 - \frac{\partial f}{\partial x_j} (x) \right) \dd{X_i} \\
&{\hskip 0.5cm}+ (f(x+X) - f(x)) \frac{\partial g}{\partial x_i} (x+X) \dd{X_j} 
-  (g(x+X) - g(x)) \frac{\partial f}{\partial x_j} (x+X) \dd{X_i} \\
&=  \left[ f \dd{x_i}, g \dd{x_j} \right] \\
&{\hskip 1cm}+ \left( f(x+X)  \frac{\partial g}{\partial x_i} (x+X) - f(x)  \frac{\partial g}{\partial x_i} (x) \right)
 \dd{X_j} \\
&{\hskip 1cm} - \left( g(x+X)  \frac{\partial f}{\partial x_j} (x+X) - g(x)  \frac{\partial f}{\partial x_j} (x) \right)
 \dd{X_i} \\
& =  \varphi \left(1 \otimes \left[ f \dd{x_i}, g \dd{x_j} \right] \right)
\end{align*}
Next we are going to use the fact that $\varphi \left( g \otimes f \dd{x_i} \right) 
= g \varphi \left( 1 \otimes f \dd{x_i} \right)$. Then
\begin{align*}
&\left[ \varphi \left( g_1 \otimes f_1 \dd{x_i} \right), \varphi \left( g_2 \otimes f_2 \dd{x_j} \right)
\right]
= \left[ g_1 \varphi \left( 1 \otimes f_1 \dd{x_i} \right), g_2 \varphi \left( 1 \otimes f_2 \dd{x_j} \right)  \right] \\
& = g_1 g_2 \left[ \varphi \left(1 \otimes f_1 \dd{x_i} \right), \varphi \left(1 \otimes f_2 \dd{x_j} 
\right) \right] \\
&{\hskip 0.5cm}+ g_1 f_1 \frac{\partial g_2}{\partial x_i} \varphi \left(1 \otimes f_2 \dd{x_j} 
\right)
- g_2 f_2 \frac{\partial g_1}{\partial x_j} \varphi \left(1 \otimes f_1 \dd{x_i} 
\right) \\
&=  \varphi \left( \left[ g_1 \otimes f_1 \dd{x_i}, g_2 \otimes f_2 \dd{x_j} \right] \right).
\end{align*} 
Verification that $\psi$ defines a homomorphism 
$V \ltimes (A \wo \LL_+) \rightarrow \hJ$
is identical to the proof of Lemma 3.4 in \cite{BIN}.

Let us verify that $\psi \circ \varphi = \id$.
\begin{align*}
&\psi \circ \varphi \left(g \otimes f \dd{x_i} \right) 
= \psi \left( gf \dd{x_i} \right) 
+ \sum_{k \in \Z_{\geq 0}^N \backslash \{ 0 \}} \frac{1}{k!} \, 
\psi \left( g \frac{\partial^k f}{\partial x^k} \otimes X^k \dd{X_i} \right) \\
&= (g \otimes 1)  \sum_{k \in \Z_{\geq 0}^N} \frac{(-1)^k}{k!} \, 
\frac{\partial^k f}{\partial x^k} \otimes \delta(x)^k \dd{x_i} 
= g \otimes f \dd{x_i}
\end{align*}
by Proposition \ref{quasiTaylor}.
Verification of $\varphi \circ \psi = \id$ is the same as in Lemma 3.5 of \cite{BIN}. 

For part (b), we point out that the ideal $\Delta^{m+1} \otimes_A V$ gets mapped to $A \wo \LL_{\geq m}$.
\end{proof}
 
 \begin{remark}
 In case of the affine space, the isomorphism of the above theorem also holds without completions:
$$J(\Aff^N) \cong V(\Aff^N) \ltimes (K[x_1, \ldots, x_N] \otimes \LL_+).$$ 
 \end{remark}

 \begin{corollary}
 \label{iso}
Let $U$ be an \'etale chart of $X$ with uniformizing parameters ${x_1, \ldots, x_N}$.
Let $A = \A(U)$, $V = \V(U)$, $\hJ = \JJ(U)$, $\hav = D(A, \hJ)$, $D = \D(U)$. 

(a) Then
$$\hav \cong D \mathop\otimes\limits_A \cU_A(A \wo \LL_+), $$
where  $\cU_A(A \wo \LL_+)$ is the universal enveloping algebra over ring $A$.

(b) Let $AV^m = D(A, J^m)$. Then 
$$AV^m \cong D \otimes_K \cU(\LL_+/\LL_{\geq m}).$$ 
 \end{corollary}
 \begin{proof}
 By Theorem \ref{isojet}, $\hav = D( A, \hJ) \cong D(A, V \ltimes (A \wo \LL_+))$. As a vector space, the strong
 Rinehart enveloping algebra of a semidirect product decomposes as
 $$D(A, V \ltimes (A \otimes \hLL_+)) \cong D(A, V) \otimes_A D(A, A \wo \LL_+).$$
 However $D(A, V) = D$ and $D( A, A \wo \LL_+) = \cU_A(A \wo \LL_+)$.
Note that algebras $D$ and $1 \wo \LL_+$ commute. 

Part (b) is an immediate consequence of Theorem \ref{isojet} (b).
 \end{proof}

\begin{definition}
\label{topq}
We call sheaf $\F$ {\it topologically quasi-coherent} if $\F$ admits a decreasing filtration of subsheaves
$$\F \supset \F_1 \supset \F_2 \supset \ldots$$
with $\mathop\cap\limits_m \F_m = 0$, such that the quotients $\F / \F_m$ are quasi-coherent sheaves.
\end{definition}

Sheaves of jets of functions and jets of vector fields are examples of topologically quasi-coherent sheaves.

Sheaf $\hAV$ has a decreasing filtration 
$$\hAV \supset \hAV_0 \supset \hAV_1 \supset \hAV_2 \supset \ldots,$$
where $\hAV_m$ is the kernel of the projection $\hAV \rightarrow \AV^m$. This filtration defines the completion topology on $\hAV$. 
   
 \begin{theorem}
 \label{quasicoherent}
(a) The sheaf $\AV^m$ is quasi-coherent.

\noindent
(b) The sheaf $\hAV$ is topologically quasi-coherent.
 \end{theorem}
 
 \begin{proof}
 By the previous result, for any affine subset $U$ of an  \'etale chart we have the isomorphism
 $\hAV^m(U) \cong \D(U) \otimes \cU (\LL_+/\LL_{\geq m})$. Since $\D$ is a quasi-coherent sheaf, by Theorem 
 \ref{localization}, $\hAV^m$ is quasi-coherent. Part (b) follows from part (a).
 \end{proof}
 
 An explicit localization formula in sheaves $\JJ$ and $\hAV$ for the vector fields is given by the following lemma: 
 
 \begin{proposition}
 \label{localizationformula}
 Let $U$ be an affine open subset of $X$, and let $g$ be a non-zero function in $A = \A(U)$. 
 Let $U_g$ be $U$ without the locus of zeros of $g$, and let $A_g = \A(U_g)$ be the localization of $A$ by $g$.
 Then the following relation holds in $\JJ(U_g)$ and $\hAV(U_g)$:
 $$1 \# \frac{1}{g} \eta = \sum_{k=0}^\infty \sum_{s=0}^k (-1)^s {k \choose s} 
 \frac{1}{g^{s+1}} \# g^s \eta 
 = \sum_{k=0}^\infty \left(  \frac{1}{g^{k+1}} \otimes 1 \right) \delta(g)^k \eta, 
 \ \ \text{for \ } \eta \in \V(U).$$
\end{proposition}
\begin{proof}
 Here we can immediately see that 
 $$ \left(  \frac{1}{g^{k+1}} \otimes 1 \right) \delta(g)^k \eta \in J_k (U_g),$$
 while
 $$1 \# \frac{1}{g} \eta 
 - \sum_{k=0}^m \sum_{s=0}^k (-1)^s {k \choose s}  \frac{1}{g^{s+1}} \# g^s \eta
 =  \sum_{s=0}^{m+1} (-1)^s {m+1 \choose s}  \frac{1}{g^{s}} \# g^{s-1} \eta$$
 and belongs to $J_{m+1} (U_g)$.
\end{proof}

We saw in Corollary \ref{iso} that Lie algebra $\hLL_+ = 1 \wo \LL_+$ plays an important role in the theory of $AV$-modules. 
As another corollary of Theorem \ref{isojet} we get a construction of a topological vector bundle on $X$ with fibre 
$\hLL_+$. In our next Theorem we derive the change of coordinates formula for this vector bundle. This formula underscores the need for a completion in the setting of sheaves.

Let $U_1$ be an \'etale chart in $X$ with uniformizing parameters $\{ x_1, \ldots, x_N \}$ and let $U_2$ be another  \'etale chart in $X$ with uniformizing parameters $\{y_1, \ldots, y_N \}$.
Suppose in the intersection $U_1 \cap U_2$ we have the following transition functions:
$$ x_i = G_i (y_1, \ldots, y_N) \in \A(U_2), \ \ y_j = H_j (x_1, \ldots, x_N) \in \A(U_1), \ \ 
i,j = 1,\ldots, N.$$
Here we are slightly abusing the notations since $H_i$ and $G_j$ are not polynomials in the respective variables. What we are emphasizing with these notations is the fact that we can take partial derivatives of these functions in the corresponding variables.

\begin{definition}
We call a sheaf $\F$ on $X$ a {\it topological vector bundle} if $\F$ has a decreasing filtration 
 of subsheaves
$$\F \supset \F_1 \supset \F_2 \supset \ldots$$
with $\mathop\cap\limits_m \F_m = 0$, and there exists an open cover of $X$ by affine subsets $\{ U_i \}$,
such that the quotients $\F / \F_m$ have compatible trivializations in each open set $U_i$:
$\F / \F_m (U_i) \cong \O (U_i) \otimes V_{m,i}$, where projections 
$\F / \F_{m+1} (U_i) \rightarrow \F / \F_m (U_i)$ are induced by maps $V_{m+1, i} \rightarrow V_{m,i}$.
\end{definition}
The fibre of a topological vector bundle $\F$ over $U_i$ is $\varprojlim V_{m,i}$. 

\begin{theorem}
\label{bundle}
There exists a topological vector bundle on $X$ with fibre $\hLL_+ = 1 \wo \LL_+$ with the transition functions between two 
\'etale charts $U_1$, $U_2$ given by the formula:
\begin{align}
&X^m \dd{X_p} \mapsto \notag \\ 
\label{Ltransform}
&\hbox{\hskip 1cm}
\sum_{0 \leq k \leq m} (-1)^{m-k} {m \choose k} G(y)^{m-k} \times \\
&\hbox{\hskip 2cm}
\sum_{q=1}^N \left( G(y+Y)^k \frac{\partial H_q}{\partial x_p} (G(y+Y))
- G(y)^k \frac{\partial H_q}{\partial x_p}(G(y)) \right)\dd{Y_q} \notag
\end{align}
and extended to the power series by continuity.
\end{theorem}


This formula is obtained by passing from $\hLL_+$ to $\JJ(U_1)$, performing the transition to 
$\JJ(U_2)$ and returning to $\A(U_2) \wo \LL_+$. We omit the details of this straightforward calculation.

It is easy to see that in (\ref{Ltransform}) the terms with $Y^s$ where $|s| < |m|$, vanish. This means that this infinite-dimensional vector bundle has a filtration of subbundles corresponding to spaces $\hLL_{\geq k}$, which induce finite-dimensional quotient bundles.  
The quotient bundle corresponding to $\hLL_+ / \hLL_{k \geq 1}$ is the vector bundle of 
$(1,1)$-tensors $\Omega^1 \otimes_{\A} \V$.
 



\bibliographystyle{alpha}
\bibliography{bibs}

\end{document}